\documentclass[12pt, oneside]{article}   	
\usepackage{geometry}
\usepackage{graphicx}				

\usepackage{indentfirst}
\usepackage{hyperref}
\usepackage{tikz-cd}
\usepackage{ulem}	
\usepackage{amsthm}
\usepackage{amsfonts}
\usepackage{amsmath}
\usepackage{amscd}
\usepackage{enumerate}
\usepackage{amssymb}	
\usepackage{verbatim}
\usepackage{cancel}
\usepackage{cite}
\usepackage{textcomp}								
\usepackage{color}
\usepackage{mathrsfs}
\usepackage{fourier}

\newtheorem{example}{Example}[section]							
\newtheorem{theorem}{Theorem}[section]

\newtheorem{proposition}{Proposition}[section]
\newtheorem{lemma}{Lemma}[section]
\newtheorem{claim}{Claim}[section]
\newtheorem{remark}{Remark}[section]

\newtheorem{thmm}{Theorem}

\def\T{\mathbb T}
\def\N{\mathbb N}
\def\R{\mathbb R}\def\Z{\mathbb Z}

\def\cC{\mathcal C}

\def\be{\begin{equation}}
\def\ee{\end{equation}}

\def\cX{\mathcal X}
\def\cF{\mathcal F}
\def\cZ{\mathcal Z}

\title{Global rigidity of smooth $\Z\ltimes_\lambda\R$-actions on $\T^2$}
\author{Changguang Dong \quad and \quad Yi Shi}
\date{\today}

\begin{document}
\maketitle
\begin{abstract}
For $\lambda>1$, we consider the locally free $\Z\ltimes_\lambda\R$ actions  on $\T^2$. We show that, if the action is $C^r$ with $r\geq2$, then it is $C^{r-\epsilon}$-conjugate to an affine action generated by a hyperbolic automorphism and a linear translation flow along expanding eigen-direction of the  automorphism. In contrast,  there exists a $C^{1+\alpha}$-action which is semi-conjugate, but not topologically conjugate to an affine action.
\end{abstract}


\section{Introduction}

In dynamical systems, rigidity phenomena have  been extensively studied during past decades. Especially, there are tremendously amount of advancements, with fascinated applications  not only in dynamical systems, but also   in other fields such as geometry, number theory etc. These include but are not limited to, orbit closure rigidity (e.g. \cite{Fu, Be,Ra}), measure rigidity (e.g. \cite{Ra,EL1,EKL,KKRH}), local rigidity (e.g. \cite{DK1,DK2, Fi}), global rigidity (e.g. \cite{SV,DSVX}) etc. 
This article is a contribution to the global rigidity program. Namely, we would like to understand, describe and classify all the actions of a specific group on a specific manifold. 

The group we are considering is a special type of so-called abelian-by-cyclic groups. It is given as follows. Let $\lambda>1$, and $G_\lambda:=\Z\ltimes_\lambda\R$ defined by the following group relation$$(m,t)\circ (n,s)=(m+n,\lambda^ms+t),\text{ for any }(m,t),(n,s)\in\Z\times\R.$$

We are interested in  the  action of $G_\lambda$ on $2$-torus $\T^2$ by smooth diffeomorphisms. 
Typical examples of such actions are affine actions. Let $A\in{\rm GL}(2,\Z)$ with an eigenvalue $\lambda>1$ and corresponding unit eigenvector $v$ with $Av=\lambda v$. Then for every constant $a>0$, it is easy to see that the automorphism $A$ together with the flow generated by $av$ (flow direction is $v$ with constant velocity $a$) on $\T^2$ generates the group $G_\lambda$, hence this  gives an affine solvable action of $G_{\lambda}$. 

One may wonder, whether there exist other smooth $G_\lambda$ actions on $\T^2$ up to smooth conjugacy.
To answer this question, we show that all $C^{2+}$-smooth actions of $G_{\lambda}$ on $\T^2$ is smoothly conjugate to an affine action.

\begin{thmm}\label{th:m-1}
Let $\lambda>1$ and $r\geq2$. Suppose  $\rho:\Z\ltimes_\lambda\R\to \operatorname{Diff}^r(\T^2)$ is a locally free action, then it is $C^{r-\epsilon}$-smoothly conjugate to an affine action, for any $\epsilon>0$. More precisely, there exist
\begin{itemize}
	\item a hyperbolic automorphism $A\in{\rm GL}(2,\Z)$ where $\lambda$ is the unstable eigenvalue of $A$;
	\item a flow $v_t$ generated by the unit unstable vector field of $A$;
\end{itemize}
such that $\rho$ is $C^{r-\epsilon}$-smoothly conjugate to the group action generated by $\{A,v_{at}\}$ for some $a\in\R\setminus\{0\}$.
In particular, if the action $\rho$ is $C^\infty$-smooth, then it is $C^\infty$-smoothly conjugate to an affine action.
\end{thmm}

Recently there is an increasing interest in the study of rigidity properties for actions of abelian-by-cyclic groups, see \cite{As,BW,BMNR,HX,WX,Liu}. In the literature there are much more attention in higher rank abelian group actions as well as the well-known Zimmer program (classifying actions of higher rank Lie groups/lattices). This may be the case as explained in the following. The abelian actions have lots of symmetry either along  Lyapunov foliations in the ambient space or  from the  structure of  acting group, and more importantly there are many deep applications in diophantine approximation etc.  As for the latter one, Zimmer program is aiming to classify higher rank Lie groups/lattices acting on low dimensional manifolds, which brings together many fields such as group theory, dynamics and rigidity etc. In contrast, the group we consider here does not seem to have certain properties like symmetry or rigidity (or super rigidity), so it is commonly known that in general one should not expect any rigidity phenomenon for such group actions. Nevertheless, it is quite surprising, as we state in Theorem \ref{th:m-1}, that when restricted to some special manifolds (say $\T^2$), rigidity result is still possible to be obtained.

There are a few interesting work that are related to ours. We list several of them here. By considering the same acting group, in \cite{Liu} a local rigidity result on $\T^d$ is proven under some additional conditions. And in the Lie group setting \cite{Wa,As}, the authors obtained a few local rigidity result for certain special solvable group actions, under different conditions. We also note that in \cite{WX,HX,BW,BMNR}, the authors studied various discrete abelian-by-cyclic group actions, and showed certain local rigidity. Additionally, we refer to \cite{HX,As1,Fi} and the references therein for more details about these work.

 Now let us explain briefly our argument. We obtain certain hyperbolicity by combining the Denjoy's theory for circle maps and the geometry of the invariant foliations, then via Franks \cite{Fr}, we get topological conjugacy. After that, we obtain the rigidity of Lyapunov exponents, which can be approximated by those on periodic orbits. From here we can complete the proof by using  \cite{Jo, dlL1,dlL2}. Our technique shares some similarity to \cite{WX}, however neither do we use KAM iterative scheme, nor assume a priori any  hyperbolicity of the action or  diophantine condition on the rotation number (vector), all of which are heavily relied on in \cite{WX,Liu}. Let's remark that, to extend our argument in the higher dimensional manifold, a result analogous to  Herman's result for pseudo rotations seems to be necessary.

At last, we would like to emphasize that the regularity assumption, i.e. $r\ge 2$, is crucial in the proof. In particular, the assertion in Theorem \ref{th:m-1} can not be obtained if the action $\rho$ is only $C^{1+\alpha}$-smooth for some $\alpha\in[0,1)$. We have the following example of a $C^{1+\alpha}$ $G_\lambda$-action which is not topologically conjugate to a linear model. 

\begin{example}
	Let $A\in{\rm GL}(2,\Z)$ be a hyperbolic automorphism on $\T^2$. We can do the DA-construction in a small neighborhood of a fixed point of $A$, see \cite[(9.4d)]{Sm} and \cite[Chapter 8.8]{Ro}. There exists a diffeomorphism $f:\T^2\to\T^2$ satisfying:
	\begin{itemize}
		\item $f$ is partially hyperbolic $T\T^2=E^{cs}\oplus E^u$ admitting $C^{1+\alpha}$-smooth unstable foliation $\cF^u$ tangent to $E^u$ and linear foliation tangent to $E^{cs}$ which is the stable foliation of $A$; 
		\item the $\Omega$-set of $f$ consists of a source and a hyperbolic expanding attractor.
	\end{itemize}
	The fact that $f$ preserves the linear stable foliation of $A$ implies that $\|Df|_{E^u(x)}\|=\lambda$ for every $x\in\T^2$ by taking an adapted metric. Let $\phi_t$ be a flow going through the unstable foliation $\cF^u$ of $f$ with constant flow speed preserving the linear stable foliation of $A$, then the action  $\rho:\Z\ltimes_\lambda\R\to \operatorname{Diff}^{1+\alpha}(\T^2)$ defined by
	$$
	f=\rho(1,0)
	\qquad \text{and} \qquad
	\phi_t=\rho(0,t) 
	\qquad \text{satisfies} \qquad
	f\circ \phi_t=\phi_{\lambda t}\circ f.
	$$
	It is clear that this action is not topologically conjugate to any affine actions.
\end{example}

\section{Invariant foliation and linear action}

Let $f=\rho(1,0)$ and $\phi_t=\rho(0,t)$. By the group relation, \be\label{eq:g-rel}
f\circ \phi_t=\phi_{\lambda t}\circ f.
\ee
Let $\cX$ be the vector field generating $\phi_t$, namely $$\cX(x)=\frac{d}{dt}|_{t=0}\phi_t(x).$$Notice that by our assumption, $\cX$ is a smooth vector field and $\cX(x)\neq0$ for every $x\in\T^2$.

We have the following important observation.

\begin{lemma}\label{le:uni}
There exists a constant $C\ge 1$, such that for any $n\in \N$ and $x\in \T^n$, 
$$
C^{-1}\lambda^n\le \|Df^n|_{\cX(x)}\|\le C\lambda^n.
$$
In particular, if we denote $\cF^u$ the foliation generated by $\phi_t$, then for  any ergodic measure $\mu$ of $f$, the Lyapunov exponent of $f$ on $\cF^u$ is $\log\lambda$.
\end{lemma}
\begin{proof}
From \eqref{eq:g-rel}, we have $f\circ \phi_t\circ f^{-1}(x)=\phi_{\lambda t}(x)$. By taking derivatives both sides, we have $$Df|_{\cX(f^{-1}x)}\cdot \cX(f^{-1}x)=\lambda \cX(x).$$Hence $$Df^n|_{\cX(x)}\cdot \cX(x)=\lambda^n \cX(f^n(x)),$$therefore $\|Df^n|_{\cX(x)}\|= \lambda^n\frac{\|\cX(f^n(x))\|}{\|\cX(x)\|}.$ The proof is complete by setting $C=\frac{\max_x\{\|\cX(x)\|\}}{\min_x\{\|\cX(x)\|\}}$.
\end{proof}

Since $f$ is uniformly expanding along $\cF^u$, if we consider the action of $f^{-1}$, it is uniformly contracting along $\cF^u$. Thus we have the following lemma which shows that $\cF^u$ is an irrational minimal foliation on $\T^2$.

\begin{lemma}\label{le:Fu}
	The foliation $\cF^u$ generated by $\phi_t$ is a $C^r$ irrational minimal foliation on $\T^2$. Moreover, its lifting foliation $\tilde{\cF}^u$ is quasi-isometric on $\R^2$, i. e. there exist constants $a, b>0$, such that for all $x\in\R^2$ and $y\in \tilde\cF^u(x)$, 
		\begin{equation}\label{eq:QI}
			d_{\tilde\cF^u}(x,y)\le a\cdot d_{\R^2}(x,y)+b.
		\end{equation}
		Here $d_{\R^2},d_{\tilde\cF^u}$ are distance functions on $\R^2$ and leaves of $\tilde{\cF^u}$ respectively.
\end{lemma}

\begin{proof}
	According to \eqref{eq:g-rel}, $f$ is uniformly expanding along $\cF^u$, thus $f^{-1}$ uniformly contracts leaves of $\cF^u$. This implies $\cF^u$ has no  closed leaves. Otherwise, assume $\gamma^u\in\cF^u$ is a closed leaf, then the length of $f^{-n}(\gamma^u)$ tends to zero as $n\to+\infty$. By taking the subsequence $f^{-k_n}(\gamma^u)\to z\in\T^2$, which contradicts to the local tubular neighborhood of $\phi_t$ around $z\in\T^2$.
	
	Since $\cF^u$ has no closed leaves, it is a suspension of a diffeomorphism of $\mathbb{S}^1$ with irrational rotation number (\cite[Theorem 4.3.3]{HH}). Since the action $\rho$ is $C^r$-smooth for $r\geq2$, the flow $\phi_t$ is $C^r$ and $\cF^u$ is minimal due to Denjoy's theorem. Moreover, the lifting foliation $\tilde\cF^u$ on $\R^2$ of a suspension foliation is quasi-isomeric. This proves the lemma.
\end{proof}



Let $f_*:H_1(\T^2)\to H_1(\T^2)$ be the induced map of $f$ on the first homology group of $\T^2$, and $F:\R^2\to \R^2$ be a lift of $f$ to the universal cover. Then $f_*:=A\in{\rm GL}(2,\Z)$ and $F(x)=Ax+G(x)$ for every $x\in\R^2$, where $G:\R^2\to \R^2$ is a $\Z^2$-periodic continuous function:
$$
G(x+m)=G(x),\qquad\forall x\in\R^2,~\forall m\in\Z^2.
$$ 
In particular, there exists $C_0>0$, such that 
$\|G(x)\|\leq C_0$ for every $x\in\R^2$.

\begin{lemma}\label{le:hyp}
The induced map $f_*=A\in{\rm GL}(2,\Z)$ is hyperbolic. 
\end{lemma}
\begin{proof}
Assume $A$ is not hyperbolic. For any bounded set $\gamma\subset\R^2$, denote 
$\|\gamma\|=\sup_{x\in\gamma}\{\|x\|\}$, 
then we inductively have
\begin{align*}
	\|F(\gamma)\|&=A(\gamma)+G(\gamma)
	\leq\|A\|\cdot\|\gamma\|+C_0\\
	\|F^2(\gamma)\|&=\|F(A(\gamma)+G(\gamma))\|
	\leq\|A^2(\gamma)+A\circ G(\gamma)\|+C_0 \\
	&\leq\|A^2\|\cdot\|\gamma\|+C_0\|A\|+C_0\\
	&\cdots \cdots \\
	\|F^k(\gamma)\|&\leq\|A^k\|\cdot\|\gamma\|+
	C_0\cdot\left(\sum_{i=0}^{k-1}\|A^i\|\right).
\end{align*}
Since $A$ is not hyperbolic, $\|F^k(\gamma)\|$ has at most polynomial growth rate in $\R^2$ with respect to $k$. 

However, if we take a segment $\gamma^u\subset\tilde{\cF}^u(x)$ for any $x\in\R^2$ with two endpoints $x,y\in\gamma^u$, then
$$
d_{\tilde{\cF}^u}(F^n(x),F^n(y))\geq 
C^{-1}\lambda^n\cdot d_{\tilde{\cF}^u}(x,y)
$$
Lemma \ref{le:Fu} shows that $\tilde{\cF}^u$ is quasi-isometric, thus we have
$$
d_{\R^2}(F^n(x),F^n(y))~\geq~
\frac{1}{a}\cdot\left(d_{\tilde{\cF}^u}(F^n(x),F^n(y))-b\right)
~\geq~
\frac{1}{a}\cdot\left(C^{-1}\lambda^n\cdot d_{\tilde{\cF}^u}(x,y)-b\right).
$$
This implies
$$
\|F^n(\gamma^u)\|~\geq~
\frac{1}{2}\cdot d_{\R^2}(F^n(x),F^n(y))
~\geq~
\frac{1}{2a}\cdot\left(C^{-1}\lambda^n\cdot d_{\tilde{\cF}^u}(x,y)-b\right),
$$
which has exponential growth rate. This is a contradiction, so $f_*=A\in{\rm GL}(2,\Z)$ is hyperbolic.
\end{proof}

\section{Topological conjugacy on $\T^2$}

Since the induced map $f_*=A\in{\rm GL}(2,\Z)$ is hyperbolic, J. Franks \cite{Fr} proved the following semi-conjugacy for the action of $f$.

\begin{theorem}[\cite{Fr}]\label{th:franks}
	There exists a continuous surjective map $H:\R^2\to \R^2$ such that
	\begin{itemize}
		\item $H(x+m)=H(x)+m$ for any $x\in\R^2$ and $m\in\Z^2$;
		\item there exists a constant $K>0$, such that $\|H-{\rm Id}\|_{C^0}<K$;
		\item $H\circ F(x)=A\circ H(x)$ for any $x\in\R^2$.
	\end{itemize}
	Moreover, let $h:\T^2\to\T^2$ be the projection of $H$ on $\T^2$, then $h$ is continuous, surjective and satisfies $h\circ f=A\circ h$ on $\T^2$.
\end{theorem}

Applying in our context, we show the following

\begin{proposition}\label{pr:top}
With notations $H, h$ as in Theorem \ref{th:franks}, the semi-conjugacy $h:\T^2\to\T^2$ is a homeomorphism, thus a topological conjugacy. Moreover, the conjugacy $h$ maps the foliation $\cF^u$ generated by $\phi_t$ to the linear expanding foliation $L^u$ of $A$ on $\T^2$: $h(\cF^u)=L^u$. 
\end{proposition}

\begin{proof}
Since $H:\R^2\to\R^2$ is $\Z^2$-periodic which induced a continuous surjective map $h:\T^2\to\T^2$, we only need to $H$ is injective which will guarantee that both $H$ and $h$ are homeomorphisms

We have the following claim shows that $h(\cF^u)=L^u$.

\begin{claim}\label{clm:u-homeo}
Let $\tilde L^u$ be the expanding line foliation of $A$ on $\R^2$. For any $x\in \R^2$, the map $H$ satisfies
$$
H(\tilde\cF^u(x))=\tilde L^u(H(x)),
\qquad \text{and} \qquad
H:\tilde\cF^u(x)\to \tilde L^u(H(x))
$$ 
is a homeomorphism.
\end{claim}

\begin{proof}[Proof of the Claim]
First we have $H(\tilde \cF^u(x))\subset \tilde L^u(H(x))$. For any $y\in\tilde\cF^u(x)$, by \eqref{eq:g-rel}, 
$$
d_{\R^2}\big(F^{-n}(x),F^{-n}(y)\big)\to 0,\quad \text{as }n\to +\infty.
$$
Together with $\|H-{\rm Id}\|_{C^0}<K$, this implies that there exists $C>0$ such that $$
d_{\R^2}\big(H\circ F^{-n}(x),H\circ F^{-n}(y)\big)<C, \quad\forall n\geq0.
$$
Then by the semi-conjugacy $H\circ F(x)=A\circ H(x)$, we have 
$$
d_{\R^2}\big( A^{-n}(H(x)), A^{-n}(H(y))\big)<C.
$$
Hence $H(y)\in \tilde L^u(H(x))$.

Moreover, $H:\tilde\cF^u(x)\to \tilde L^u(H(x))$ is injective. Actually for any $y,z\in\tilde\cF^u(x)$, since $\tilde\cF^u$ is quasi-isometric, 
$$
d_{\R^2}\big(F^n(y),F^n(z)\big)\geq
\frac{1}{a}\big(d_{\tilde\cF^u}(F^n(y),F^n(z))-b\big)
\to\infty,
\quad\text{as}~n\to+\infty.
$$
If $H(y)=H(z)$, then $H\circ F^n(y)=A^n\circ H(y)=A^n\circ H(z)=H\circ F^n(y)$ and 
$$
d_{\R^2}(F^n(y), F^n(z))\leq 
d_{\R^2}(F^n(y), H\circ F^n(y))+d_{\R^2}(H\circ F^n(z), F^n(y))\leq 2K,
\qquad \forall n\geq0.
$$
This is a contradiction, thus $H(y)\neq H(z)$ and $H:\tilde\cF^u(x)\to \tilde L^u(H(x))$ is injective.

Finally, the facts that both $\tilde \cF^u(x)$ and $ \tilde L^u(H(x))$ are simply connected and $\|H-{\rm Id}\|_{C^0}<K$ show $H$ is surjective, thus $H(\tilde\cF^u(x))=\tilde L^u(H(x))$ and $H:\tilde\cF^u(x)\to \tilde L^u(H(x))$ is a homeomorphism for every $x\in\R^2$ as claimed.
\end{proof}

To complete the proof, we consider the corresponding quotient maps. Recall that, $F:\R^2\to\R^2$ preserving the foliation $\tilde\cF^u$ and $A:\R^2\to \R^2$ preserving $\tilde L^u$. Since $H$ preserves expanding foliations $H(\tilde\cF^u(x))=\tilde L^u(H(x))$ for every $x\in\R^2$, the commutative diagram $H\circ F=A\circ H$ reduced to a diagram of the corresponding quotient maps on quotient spaces $\R^2/\tilde\cF^u$ and $\R^2/\tilde L^u$. Namely, we have the following diagram
\be\label{eq:c-dia1}
\begin{tikzcd}
\R^2/{\tilde\cF^u}\arrow{r}{\hat F} \arrow[swap]{d}{\hat H} & \R^2/{\tilde\cF^u} \arrow{d}{\hat H} \\
\R^2/{\tilde L^u} \arrow{r}{\hat A} & \R^2/{\tilde L^u}
\end{tikzcd}
\ee
where $\hat F$ ($\hat A, \hat H$ respectively) is induced by $F$ ($A, H$ respectively) on the quotient space. Since both $\cF^u$ and $L^u=\pi(\tilde L^u)$ are irrational minimal foliation on $\T^2$, both quotient spaces $\R^2/\tilde\cF^u$ both $\R^2/\tilde L^u$ are necessarily isomorphic to $\R$. We denote $\R_{\tilde\cF^u}=\R^2/\tilde\cF^u$. Notice that the quotient space $\R^2/\tilde L^u$ is equal to  the stable leaf $\tilde L^s(0)$ of $A$ at $0\in\R^2$: $\tilde L^s(0)=\R^2/\tilde L^u$; and the quotient map $\hat A=A$ on $\tilde L^s(0)$. Thus the diagram \eqref{eq:c-dia1} induces 
\be\label{eq:c-dia2}
\begin{tikzcd}
\R_{\tilde\cF^u}=\R^2/{\tilde\cF^u}\arrow{r}{\hat F} \arrow[swap]{d}{\hat H} & 
\R_{\tilde\cF^u}=\R^2/{\tilde\cF^u} \arrow{d}{\hat H} \\
\tilde L^s(0)=\R^2/{\tilde L^u} \arrow{r}{\hat A=A~} & \tilde L^s(0)=\R^2/{\tilde L^u}
\end{tikzcd}
\ee
where $\hat A=A:\tilde L^s(0)\to\tilde L^s(0)$ is the linear contracting map $A(x)=\lambda^{-1}x$ for every $x\in\tilde L^s(0)\subset\R^2$.

We have the following claim:
\begin{claim}
	 We fix an orientation on $\R_{\tilde\cF^u}$ and the induced orientation on $\tilde L^s(0)$ by $H$, then the quotient map $\hat H:\R_{\tilde\cF^u}\to\tilde L^s(0)$ satisfies
	\begin{itemize}
		\item[({\bf a1})] $\hat H$ is orientation-preserving and increasing;
		\item[({\bf a2})] $\hat H$ is a bijection.
	\end{itemize}
\end{claim}

\begin{proof}[Proof of the Claim]
	Since $\|H-{\rm Id}\|_{C^0}<K$, the orientation of $\R_{\tilde\cF^u}$ induces an orientation on $\tilde L^s(0)$ by $H$ globally on $\R^2$. Since $F:\R^2\to\R^2$ is a diffeomorphism, the quotient map $\hat F:\R_{\tilde\cF^u}\to\R_{\tilde\cF^u}$ is a homeomorphism. By iterating $\hat F$ is necessary, we can assume  $\hat F:\R_{\tilde\cF^u}\to\R_{\tilde\cF^u}$ preserves the orientation and so does $A:\tilde L^s(0)\to\tilde L^s(0)$.
	
	For ({\bf a1}), assume otherwise that two points $\hat x, \hat y\in\R_{\tilde\cF^u}$ satisfy $\hat x<\hat y$ and $\hat H(\hat x)>\hat H(\hat y)$ in $\tilde L^s(0)$. Since $\hat F$ preserves the orientation, for $\tilde\cF^u(x)=\hat x$ and $\tilde\cF^u(y)=\hat y$, we have
	$$
	\hat F^{-n}(\hat x)<\hat F^{-n}(\hat y),
	\qquad \text{and} \qquad
	F^{-n}(\tilde\cF^u(x))<F^{-n}(\tilde\cF^u(y)),
	\qquad \forall n>0.
	$$
	
	For $\hat H(\hat x)>\hat H(\hat y)$ and $\hat H(\hat x)=\hat L^u(H(x))>\hat H(\hat y)=\hat L^u(H(y))$, since $A^{-1}$ is uniformly expanding along $\tilde{L}^s(0)$, we have
	$$
	\hat A^{-n}(\hat H(\hat x))-\hat A^{-n}(\hat H(\hat y))\to +\infty, \qquad \text{as}~n\to+\infty.
	$$
	This is equivalent to the Hausdorff distance between $A^{-n}(\tilde L^u(H(x)))$ and $A^{-n}(\tilde L^u(H(y)))$ tends to infinity as $n\to+\infty$ and $A^{-n}(\tilde L^u(H(x)))>A^{-n}(\tilde L^u(H(y)))$. 
	
	However, since $F^{-n}(\tilde\cF^u(x))<F^{-n}(\tilde\cF^u(y))$ for every $n>0$, the semi-conjugation
	$$
	A^{-n}(\tilde L^u(H(x)))=H\circ F^{-n}(\tilde\cF^u(x)),
	\qquad 
	A^{-n}(\tilde L^u(H(y)))=H\circ F^{-n}(\tilde\cF^u(y)),
	$$
	 and $\|H-{\rm Id}\|_{C_0}<K$ shows that $A^{-n}(\tilde L^u(H(x))$ has $2K$-bounded distance with the negative component of $\R^2\setminus A^{-n}(\tilde L^u(H(y)))$ for every $n>0$. This contradicts to the Hausdorff distance between $A^{-n}(\tilde L^u(H(x)))$ and $A^{-n}(\tilde L^u(H(y)))$ tends to infinity as $n\to+\infty$ and $A^{-n}(\tilde L^u(H(x)))>A^{-n}(\tilde L^u(H(y)))$. This proves ({\bf a1}).
	 
	 For ({\bf a2}), the surjective part of $\hat H$ comes from $H:\R^2\to\R^2$ is surjective. We only need to show $\hat H$ is injective. Otherwise, $\hat H(\hat x)=\hat H(\hat y)$, meaning that both $\tilde \cF^u(x)$ and $\tilde\cF^u(y)$ are mapped to a single line $\tilde L^u(H(x))=\tilde L^u(H(y))$. Since $\hat H$ is orientation-preserving and increasing, it follows that $H$ maps the the region $R_{x,y}\subset \R^2$ bounded by $\tilde \cF^u(x)$ and $\tilde\cF^u(y)$ in $\R^2$, to $\tilde L^u(H(x))$. However, since $\tilde\cF^u$ is an irrational minimal foliation, we have 
	 $$
	 \T^2=\pi\big(R_{x,y}\big)
	 \qquad \text{and} \qquad
	 \T^2=h(\T^2)=h\circ\pi\big(R_{x,y}\big))=
	 \pi\big(H(R_{x,y})\big)=\pi\big(\tilde L^u(H(x))\big).
	 $$
	 This is contradiction since $\pi\big(\tilde L^u(H(x))\big)$ is a single 1-dimensional leaf in $\T^2$. This proves the claim.
\end{proof}
Finally, $H$ is injective follows from $H|_{\tilde\cF^u(x)}$ is injective for every $x\in\R^2$ and $\hat H$ is injective. Thus $H:\R^2\to\R^2$ is a homeomorphism.
\end{proof}

\section{Smooth conjugacy}

By Proposition \ref{pr:top}, $f:\T^2\to\T^2$ is topologically conjugate to the hyperbolic automorphism $A=f_*\in{\rm GL}(2,\Z)$ by a homeomorphism $h:\T^2\to\T^2$ where $h\circ f=A\circ h$. 

By Lemma \ref{le:uni}, $f$ is uniformly expanding with constant Lyapunov exponent $\log\lambda$ for every ergodic measure along the $C^2$-foliation $\cF^u$, which is the orbit foliation of $\phi_t$.
Moreover, the conjugacy $h:\T^2\to\T^2$  satisfies $h(\cF^u)=L^u$.  
So we have the following lemma.

\begin{proposition}\label{pr:ph-splitting}
	The diffeomorphism $f:\T^2\to\T^2$ is partially hyperbolic
	$T\T^2=E^{cs}\oplus E^u$ with $E^u=T\cF^u$.
	Moreover, we have
	\begin{itemize}
		\item for every ergodic measure $\mu$ of $f$, the Lyapunov exponent of $\mu$ along $E^{cs}$ is non-positive;
		\item there exists an $f$-invariant foliation $\cF^{cs}$ tangent to $E^{cs}$, and the conjugacy $h$ maps $\cF^{cs}$ to the linear stable foliation $L^s$ of $A$.
	\end{itemize}
\end{proposition}

\begin{proof}
	By Lemma \ref{le:uni}, for every periodic point $p$ of $f$, it has one Lyapunov exponent along $\cF^u$ is $\log\lambda$. We denote it as $\lambda^u(p)=\log\lambda$. By Proposition \ref{pr:top}, $f$ is topologically conjugate to $A$ by $h$ and $h(\cF^u)=L^u$. Since $A$ is uniformly contracting along the transversal direction of $L^u$, $f$ is topologically contracting in the transversal direction of $\cF^u$. Thus $p$ has another Lyapunov exponent $\lambda^{cs}(p)\leq 0$. 
	
	Moreover, the periodic measures of $A$ are dense in the space of ergodic measures of $A$. By the topological conjugacy, the periodic measures of $f$ are also dense in the space of ergodic measures of $f$. Thus for every ergodic measure $\mu$ of $f$, it has two Lyapunov exponents
	$$
	\lambda^{cs}(\mu)\leq 0<\lambda^u(\mu)=\log\lambda.
	$$
	
	Now we only need to show $f$ admits a dominated splitting,  which implies that $f$ is partially hyperbolic. That is the following claim.
	
	\begin{claim}
		There exists a $Df$-dominated splitting $T\T^2=E^{cs}\oplus E^u$ with $E^u=T\cF^u$, i. e. the splitting is continuous, $Df$-invariant and there exist two constants $0<\eta<1, C>1$, such that 
	    $$ 
	    \frac{\|Df^n|_{E^{cs}(x)}\|}{\|Df^n|_{E^u(x)}\|}
	    \leq C\cdot\eta^n,
	    \qquad \forall x\in\T^2,~n\geq0.
	    $$ 
	\end{claim}
	\begin{proof}[Proof of the Claim]
		The proof follows exactly the same as \cite[Proposition 5.9]{GS}.
		Since $\cF^u$ is a $C^2$-foliation, the $Df$-invariant $E^u=T\cF^u$ is a $C^1$-bundle. We have a $C^1$-smooth splitting $T\T^2=E^u\oplus E^{\perp}$ where $E^{\perp}$ is perpendicular to $E^u$. We take continuous families of unit vectors in $\{e^u(x),e^{\perp}(x)\}_{x\in\T^2}$ in $E^u,E^{\perp}$ respectively, which forms a $C^1$ base on $T\T^2$. 
		
		Since $F$ is $C^2$-smooth, there exist three families of $C^1$-functions $\{A(x)\}_{x\in\T^2}$, $\{B(x)\}_{x\in\T^2}$ and $\{C(x)\}_{x\in\T^2}$, such that in the base $\{e^u(x),e^{\perp}(x)\}_{x\in\T^2}$,
		\begin{equation*}
			Df(x)~=~\left(  
			\begin{array}{cc}
				A(x) & B(x) \\
				0    & C(x) \\
			\end{array}
			\right)
			\qquad \forall x\in\T^2.
		\end{equation*}
		Then we have
		$$
		Df(e^u(x))=A(x)e^u(fx),
		\qquad {\rm and} \qquad
		{\it proj}^{\perp}\circ Df(e^{\perp}(x))=C(x)e^{\perp}(fx),
		$$
		where ${\it proj}^{\perp}:T\T^2\to E^{\perp}$ is the projection through $E^u$.
		
		For every $x\in\T^2$ and $n\geq1$, we introduce the following notation for the cocycles
		$$
		A^n(x)=\prod_{i=0}^{n-1}A(f^i(x))
		\qquad {\rm and} \qquad 
		C^n(x)=\prod_{i=0}^{n-1}C(f^i(x)).
		$$ 
		Lemma \ref{le:uni} shows that there exists $C>1$, such that 
		$|A^n(x)|\geq C^{-1}\lambda^n$ for every $x\in\T^2$ and $n\in\N$. Moreover, for every $\epsilon>0$ and every periodic point $p$ of $f$, since the other Lyapunov exponent of $p$ is non-positive, we have
		$$
		\lim_{n\to+\infty}\frac{1}{n}\log|C^n(p)|\leq\epsilon.
		$$
		Since periodic measures are dense in all invariant measures, we have
		$$
		\lim_{n\to+\infty}\frac{1}{n}\log|C^n(x)|\leq\epsilon,
		\qquad \forall x\in\T^2
		$$
		
		We fix $0<\epsilon\ll\log\lambda$, Theorem 1.3 of~\cite{K} shows that there exists some $N=N(\epsilon)$, such that
		$$
		|C^n(x)|\leq\exp(n\epsilon),
		\qquad \forall n\geq N,~\forall x\in\T^2.
		$$
		Finally, since $B(x)$ varies $C^1$-smooth with respect to $x\in\T^2$ and uniformly bounded, there exists a continuous cone-field $\{\cC(x)\}_{x\in\T^2}$ containing $E^u$, such that 
		$$
		Df\big(\overline{\cC(x)}\big)~\subset~\cC(f(x)), \qquad \forall x\in\T^2.
		$$
		Therefore, by the cone-field criterion \cite[Theorem 2.6]{CP}, there exists a dominated splitting 
		$$
		T\T^2=E^{cs}\oplus E^u
		\qquad
	    \text{with} \qquad
	    T\cF^u=E^u.
		$$
		This proves the claim.
	\end{proof}
	Finally, Proposition 4.A.7 of \cite{Po} shows that a partially hyperbolic diffeomorphism on $\T^2$ is dynamically coherent, i. e. there exists an $f$-invariant foliation $\cF^{cs}$ tangent to $E^{cs}$. Moreover, by the topological conjugacy $h\circ f=A\circ h$, the foliation $h(\cF^{cs})$ is $A$-invariant and transverse to $L^u=h(\cF^u)$, which is unique and $L^s=h(\cF^{cs})$.
\end{proof}

\begin{remark}
	The fact $L^s=h(\cF^{cs})$ directly implies the foliation $\cF^{cs}$ is topologically contracting by $f$, i. e. for every segment $\gamma\subset\cF^{cs}(x)$ for some $x\in\T^2$, the length $|f^n(\gamma)|\to0$ as $n\to+\infty$ .
\end{remark}

The following proposition shows that if the unstable foliation of $\cF^u$ is $C^2$, then $f$ is uniformly contracting along $E^{cs}$ with constant Lyapunov exponent $-\log\lambda$. The proof is almost the same as \cite{Gu}, see also \cite{Gh, PR}, we include the proof for completeness.

\begin{proposition}\label{pr:n-exp}
For every periodic point $p$ of $f$, the Lyapunov exponent $\lambda^{cs}(p)$ of $f$ along $E^{cs}$ is equal to $-\log\lambda$. In particular, the diffeomorphism $f\in{\rm Diff}^r(\T^2)$ with $r\geq2$ is Anosov and the conjugacy $h:\T^2\to\T^2$ with $h\circ f=A\circ h$ is $C^{r-\epsilon}$-smooth.
\end{proposition}
\begin{proof}
First of all, let $\mu_{\rm max}$ be the measure with maximal entropy of $f$, which is also the measure with maximal entropy of $f^{-1}$. Since $f$ is topologically conjugate to $A$, so the measure entropy of $\mu_{\rm max}$ associated to $f^{-1}$ is equal to the topological entropy of $A$ which is $\log\lambda$. From Ruelle's inequality, the largest Lyapunov exponent of $f^{-1}$ in $\mu_{\rm max}$ satisfies
$$
\log\lambda~\leq~\lambda^+(\mu_{\rm max},f^{-1})~=~-\lambda^{cs}(\mu_{\rm max},f).
$$
From the density of periodic measures, there exists a sequence of periodic points $p_n$ whose periodic measures converge to $\mu_{\rm max}$. Then we have
\be\label{eq:small}
\lim_{n\to\infty}\lambda^{cs}(p_n)~=~
\lambda^{cs}(\mu_{\rm max},f)~\leq~-\log\lambda.
\ee
In particular,  $\lambda^{cs}(p_n)<0$ and $p_n$ is hyperbolic for $n$ large enough.

\begin{claim}
	For every pair of hyperbolic periodic points $p,q\in{\rm Per}(f)$, we have 
	$\lambda^{cs}(p)=\lambda^{cs}(q)$.
\end{claim}

\begin{proof}[Proof of the Claim]
	Since $\phi_t$ is $C^2$, when restricted to a smooth transversal cross section which is diffeomorphic to $\S^1$, the induced map $\hat \phi$ is a $C^2 $ irrational rotation, whose rotation number is a degree-2 algebraic number (i.e. the largest eigenvalue $\lambda$ of $f_*$). By Herman \cite{He}, see also \cite{KO,KT}, $\hat\phi$ is bi-Lipschitz conjugate to $R_\lambda$. As a result, there exists a constant $C_2>0$ such that for any small segment $I$ that contained in a leaf of $\cF^{cs}(x)$, let ${\rm Hol}_t:\cF^{cs}(x)\to\cF^{cs}(\phi_t(x))$ be the holonomy map induced by $\phi_t$ satisfying 
	$$
	{\rm Hol}(x)=\phi_t(x)
	\qquad \text{and} \qquad
	{\rm Hol}_t(I)\subset\cF^{cs}(\phi_t(x)),
	$$ 
	then
	\be\label{eq:bi-L}
	\frac{1}{C_2}\le\frac{|I|}{|{\rm Hol}_t(I)|}\le C_2.
	\ee
	Here, $|\cdot|$ is the length function. Here the constant $C_2$ is independent with $I$ and $t$.
	
	Now fix two distinct hyperbolic periodic points $p,q$, then the Lyapunov exponents of $p,q$ along $E^{cs}$ satisfy $\lambda^{cs}(p),\lambda^{cs}(q)<0$. In particular, we have both $\cF^{cs}(p)$ and $\cF^{cs}(q)$ are contained in the stable manifolds of $p$ and $q$ respectively. Denote $\pi>0$ the common period of $p$ and $q$: $f^\pi(p)=p$ and $f^\pi(q)=q$.
	
	Let $x$ be an intersecting point of $\cF^u(q)$ with the local stable manifold $\cF^{cs}_{\rm loc}(p)$, then there exists $t\in\R$ such that
	$x=\phi_t(q)\in\cF^{cs}(p)$. Moreover, we can define the holonomy map 
	$$ 
	{\rm Hol}_t:\cF^{cs}(q)\to\cF^{cs}(x)=\cF^{cs}(p),
	\qquad \text{with} \qquad
	{\rm Hol}_t(q)=x.
	$$
	
	Take another point $y\in\cF^{cs}_{\rm loc}(q)$ and $J\subset\cF^{cs}(q)$ has two endpoints $q,y$, then there exists a unique point $z={\rm Hol}_t(y)\in\cF^{cs}(x)=\cF^{cs}(p)$. Since we can take $y$ close to $q$ and $|J|$ being small, \eqref{eq:bi-L} implies $|{\rm Hol}_t(J)|$ is small and ${\rm Hol}_t(J)\subset\cF^{cs}(p)$ with endpoints $x={\rm Hol}_t(q)$ and $z={\rm Hol}_t(y)$.
	
	Since $f$ is $C^r$-smooth with $r\geq2$, $E^{cs}$ is H\"older continuous. This implies both $\cF^{cs}_{\rm loc}(p)$ and $\cF^{cs}_{\rm loc}(q)$ are 
	$C^{1+\text{H\"older}}$-smooth submanifolds. Since $f$ is uniformly contracting in $\cF^{cs}_{\rm loc}(p)$ and $\cF^{cs}_{\rm loc}(q)$, the distortion control argument shows there exist $K>0$, such that for every $n\geq0$,
	$$
	\frac{1}{K}\leq\frac{|f^{\pi n}(J)|}{\exp(\lambda^{cs}(q)\pi n)}\leq K,
	\qquad \text{and} \qquad
	\frac{1}{K}\leq\frac{|f^{\pi n}\circ{\rm Hol}_t(J)|}{\exp(\lambda^{cs}(p)\pi n)}
	\leq K.
	$$
	
	However, since both $\cF^{cs}$ and $\cF^u$ are $f$-invariant, the holonomy map ${\rm Hol}_t$ is commuting with $f$, thus for every $n>0$, there exists $t_n\in\R$, such that $$
	f^{\pi n}\circ{\rm Hol}_t(J)={\rm Hol}_{t_n}\circ f^{\pi n}(J)
	\qquad \text{and} \qquad 
	\frac{1}{C_2}\leq\frac{|f^{\pi n}(J)|}{|{\rm Hol}_{t_n}\circ f^{\pi n}(J)|}
	=\frac{|f^{\pi n}(J)|}{|f^{\pi n}\circ{\rm Hol}_{t}(J)|}\leq C_2.
	$$
	This implies
	$$
	\frac{1}{KC_2}\leq\frac{\exp(\lambda^{cs}(p)\pi n)}{\exp(\lambda^{cs}(q)\pi n)}
	\leq KC_2,
	\qquad \forall n>0.
	$$
	Thus we must have $\lambda^{cs}(p)=\lambda^{cs}(q)$ for every pair hyperbolic periodic points $p$ and $q$.
\end{proof}

From this claim and \eqref{eq:small}, we know that
\be\label{eq:leq}
\lambda^{cs}(p)\leq-\log\lambda<0,
\ee
for every hyperbolic periodic point $p$ of $f$.

Since $f:\T^2\to\T^2$ is topologically conjugate to $A:\T^2\to\T^2$, it also satisfies the specification property in \cite{Sig}. If there exists some periodic point $p\in{\rm Per}(f)$ satisfying $\lambda^{cs}(p)=0$, then by the specification property, there exists hyperbolic periodic points of $f$ with Lyapunov exponents arbitrarily close to zero along $E^{cs}$. This is absurd since $\lambda^{cs}(p)\leq-\log\lambda$ for every hyperbolic periodic point $p$. Thus every periodic point $p$ of $f$ is hyperbolic with
$\lambda^{cs}(p)\leq-\log\lambda$.

This implies $f$ is Anosov and $\lambda^{cs}(\mu)\leq-\log\lambda$ for every ergodic measure $\mu$ of $f$. If we consider the SRB measure $\mu^-$ of $f^{-1}$, its Lyapunov exponent along $E^{cs}$ is equal to its measure entropy $h(\mu^-,f^{-1})$, which is smaller than $\log\lambda$. So we have
$$
-\lambda^{cs}(\mu^-)=\lambda^{u}(\mu^-,f^{-1})=h(\mu^-,f^{-1})
\leq h_{\rm top}(f^{-1})=\log\lambda
$$
This implies $\lambda^{cs}(p)=\lambda^{cs}(\mu^-)\geq-\log\lambda$. Combined with \eqref{eq:leq} and Lemma \ref{le:uni}, we have
$$
\lambda^{cs}(p)=-\log\lambda,
\qquad \text{and} \qquad
\lambda^u(p)=\log\lambda,
\qquad \forall p\in{\rm Per}(f).
$$

Finally, the work of de la Llave's \cite{dlL1,dlL2} indicates that the conjugacy $h$ is $C^{r-\epsilon}$-smooth when all periodic points of $f$ have the same Lyapunov exponents to $A$. 
\end{proof}

Now we can prove Theorem \ref{th:m-1}.

\begin{proof}[Proof of Theorem \ref{th:m-1}]
 We showed that by the $C^{r-\epsilon}$-smooth conjugacy 
 $$
 A=h\circ f\circ h^{-1}=h\circ\rho(1,0)\circ h^{-1}.
 $$
 Then we can define a $C^{r-\epsilon}$-smooth flow on $\T^2$:
 $$
 \psi_t=h\circ\phi_t\circ h^{-1}=h\circ\rho(0,t)\circ h^{-1}
 \qquad
 \text{satisfying}
 \qquad
 A\circ\psi_t=\psi_{\lambda t}\circ A.
 $$ 
 Moreover, we have show that $h(\cF^u)=L^u$ which maps the orbit of $\phi_t$ to the linear unstable foliation of $A$. Thus the orbit of $\psi_t$ is $L^u$. To prove Theorem \ref{th:m-1}, we only need to show that $\psi_t$ has constant velocity.
 
 Denote 
 $$
 \cZ(x)=\frac{d}{dt}|_{t=0}\psi_t(x).
 $$
 Let $p$ be a fixed point of $A$ and $x\in L^u(p)$ with $\psi_t(p)=x$ for some $t\in\R$. Then we have
 $$
 A^n\circ\psi_t(p)=\psi_{\lambda^nt}\circ A(p),
 \qquad \text{and} \qquad
 DA^n\circ D\psi_t(\cZ(p))=D\psi_{\lambda^nt}\circ DA^n(\cZ(p)).
 $$
 This implies $DA^n(\cZ(x))=D\psi_{\lambda^nt}(\lambda^n\cdot\cZ(p))$. By taking the norm, we have
 $$
\lambda^n\cdot\|\cZ(x)\|=\lambda^n\cdot\|\cZ(\psi_{\lambda^n t}(p))\|,
 \qquad \forall n\in\Z.
 $$
 Let $n\to-\infty$, we have 
 $$
 \psi_{\lambda^n t}(p)\to p
 \qquad \text{and} \qquad
 \|\cZ(\psi_{\lambda^n t}(p))\|\to\|\cZ(p)\|.
 $$
 This implies $\|\cZ(x)\|=\|\cZ(p)\|$ for every $x\in L^u(p)$.
 
 Since $L^u(p)$ is dense in $\T^2$, we have $\|\cZ(x)\|=\|\cZ(p)\|\triangleq a$ for every $x\in\T^2$. Thus $\psi_t$ is the linear flow with constant velocity. This proves Theorem \ref{th:m-1} that $\rho$ is $C^{r-\epsilon}$-smooth conjugate to the affine action $\{A,v_{at}\}$.
\end{proof}

{\bf Acknowledgement:} C. Dong is supported by Nankai Zhide Foundation and  ``the Fundamental Research Funds for the Central Universities" No. 100-63233106. Y. Shi is supported by National Key R\&D Program of China (2021YFA1001900), NSFC (12071007,
12090015) and Institutional Research Fund of Sichuan University (2023SCUNL101).

\text{\quad}\\

\textsc{Chern Institute of Mathematics and LPMC, Nankai University, Tianjin 300071 China}

{Email: dongchg@nankai.edu.cn}

\vskip5mm

\textsc{School of Mathematics, Sichuan University, Chengdu 610065 China}

{Email: shiyi@scu.edu.cn}

\end{document}